\documentclass[a4paper,11pt]{amsart}
\usepackage[english]{babel}
\usepackage{amsmath}
\usepackage{amsfonts}
\usepackage{amssymb}
\usepackage{amsthm}
\usepackage{amsxtra}
\usepackage{epsfig,graphics}
\usepackage{graphicx}
\usepackage{hyperref}

\title{Ehrhart polynomial \\and multiplicity Tutte polynomial}

\begin{document}

\newtheorem{te}{Theorem}[section]
\newtheorem{lem}[te]{Lemma}
\newtheorem{pr}[te]{Proposition}
\newtheorem{prob}[te]{Problem}
\newtheorem{co}[te]{Corollary}
\newtheorem{cj}[te]{Conjecture}
\theoremstyle{definition}
\newtheorem{re}[te]{Remark}
\newtheorem{ex}[te]{Example}

\newcommand{\faktor}{\frac}
\newcommand{\PT}{\mathcal{C}_{0}(\Phi)}
\newcommand{\Pt}{\mathcal{C}_{0}(\Theta)}
\newcommand{\Tquo}{{\mathfrak{h}}/{\langle\Phi^\vee\rangle}}
\newcommand{\sym}{\mathfrak{S}}


\newcommand{\lc}{\Lambda_C}
\newcommand{\lcb}{\overline{\Lambda_C}}
\newcommand{\cs}{\mathcal{S}}
\newcommand{\s}{\mathcal{S}}
\newcommand{\vs}{\mathcal{V_S}}
\newcommand{\vt}{\mathcal{V_T}}
\newcommand{\vsp}{\vs^p}
\newcommand{\vtq}{\vt^q}
\newcommand{\rx}{\mathcal{R}_X}
\newcommand{\pcn}{\mathbb{P}(\mathbf{N}_T(C))}
\newcommand{\pc}{\mathbb{P}_C}
\newcommand{\zx}{\mathbf{Z}_X}
\newcommand{\yx}{\mathbf{Y}_X}
\newcommand{\zt}{z^{\mathcal{T}}}
\newcommand{\dn}{\mathbf{D}_\mathcal{N}}
\newcommand{\pf}{\mathcal{P}_X (\lambda)}

\author{Michele D'Adderio}\thanks{$^*$ supported by the Max-Planck-Institut
f\"{u}r Mathematik.}
\author{Luca Moci}\thanks{$^{\dag}$ supported by a Sofia Kovalevskaya Research Prize of
Alexander von Humboldt Foundation awarded to Olga Holtz.}
\address{Max-Planck-Institut f\"{u}r Mathematik\\
Vivatsgasse 7, 53111 Bonn\\
Germany}\email{mdadderio@yahoo.it}
\address{Institut f\"{u}r Mathematik\\ Technische Universit\"{a}t
Berlin\\
Stra$\ss$e des 17. Juni, 136, Berlin\\
Germany}\email{moci@math.tu-berlin.de}

\maketitle
\begin{abstract}
We prove that the Ehrhart polynomial of a zonotope is a
specialization of the multiplicity Tutte polynomial introduced in
\cite{MoT}. We derive some formulae for the volume and the number
of integer points of the zonotope.
\end{abstract}

\section{Introduction}
Let $P$ be a convex $n-$dimensional polytope having all its
vertices in $\Lambda=\mathbb{Z}^n$. For every $q\in \mathbb{N}$,
let us denote by $qP$ the \emph{dilation} of $P$ by a factor $q$.
A celebrated theorem of Ehrhart states that the number
$\mathcal{E}_P(q)$ of integer points in $qP$ is a polynomial
function in $q$, called the \emph{Ehrhart polynomial}. A special
case is when $P$ is the \emph{zonotope} $\mathcal{Z}(X)$ generated
by a list $X$ of vectors in $\Lambda$. In \cite{MoT} we associated
to such a list a polynomial $M_X(x,y)$, called the
\emph{multiplicity Tutte polynomial.}
The definition of $M_X(x,y)$
is similar to the one of the classical Tutte polynomial, but it
also encodes some information on the arithmetics of the list $X$.

\medskip

In this paper we prove that the multiplicity Tutte polynomial
specializes to the Ehrhart polynomial of the zonotope:
$$\mathcal{E}_{\mathcal{Z}(X)}(q)= q^n M_{X}(1+1/q,1).$$

\medskip

The polynomials $M_X(x,y)$ and $\mathcal{E}_P(q)$ have positive
coefficients. In \cite{DMa} we studied $M_X(x,y)$ in the more
general framework of \emph{arithmetic matroids}, and we provided
an interpretation of its coefficients, which extends Crapo's
formula \cite{Cr}, and it is inspired by the geometry of
generalized toric arrangements.

On the other hand, the
meaning of the coefficients of the Ehrhart polynomial is still quite mysterious.
The only known facts are that the leading coefficient is
the volume of $P$, the coefficient of the monomial of degree $n-1$
is half of the volume of the boundary of $P$, and the constant
term is $1$. Significant effort has been devoted to the attempt of
finding a combinatorial or geometrical interpretation of the other coefficients.
We hope that the present work can be a step in this direction, by relating the two polynomials.

\medskip

Furthermore, our results strengthen the connection of the
polynomial $M_X(x,y)$ with the \emph{partition function} $\pf$.
For every $\lambda \in \Lambda$, $\pf$ is defined as the number of
solutions of the equation
$$\lambda= \sum_{a\in X} x_a a \; \mbox{ , with } \; x_a\in
\mathbb{Z}_{\geq 0} \text{ for every } a.$$ (Since we want this
number to be finite, the elements of $X$ are assumed to lie on the
same side of some hyperplane). The problem of computing $\pf$ goes
back to Euler, but it is still the object of an intensive research
(see for instance \cite{StP}, \cite{BV}, \cite{li}, \cite{DPV1}; see also \cite{AP}, \cite{HR}, \cite{HRX}, \cite{Le} for related combinatorial work).
This problem is easily seen to be equivalent to the problem of
counting integer points in a variable polytope.

In order to study $\pf$, in \cite{DM} Dahmen and Micchelli introduced
a space of \emph{quasipolynomials} $DM(X)$. It is a
theorem that the cone on which $\pf$ is supported decomposes in
regions called \emph{big cells}, such that on every big cell
$\Omega$, $\pf$ coincides with a quasipolynomial
$f_\Omega(\lambda)\in DM(X)$. In fact $\pf$ and
$f_\Omega(\lambda)$ coincide on a slightly larger region, which is
the Minkowky sum of $\Omega$ and $-\mathcal{Z}(X)$. In \cite{li}
this result is obtained by studying the algebraic \emph{Laplace
transform} of $\pf$, which is a function defined on the complement
of a \emph{toric arrangement} (see \cite{Mo}, \cite{Mw}, \cite{MS}).

\medskip

The polynomial $M_X(x,y)$ is related with the partition function
in at least three ways. First, many topological and combinatorial
invariants of the toric arrangement are specializations of
$M_X(x,y)$. Second, $DM(X)$ is a graded vector space, whose
Hilbert series turns out to be $M_X(1,y)$. Both these appear in
\cite{MoT}. The third connection is the subject of the present
paper: the Ehrhart polynomial of $\mathcal{Z}(X)$ is a
specialization of $M_X(x,y)$. This yields some formulae for the
number of integer points in the zonotope and in its interior,
which are stated in Corollary \ref{ECI}.

\section{Notations and recalls}

\subsection{Ehrhart polynomial}

Let $P\subset\mathbb{R}^n$ be a convex $n-$dimensional polytope
having all its vertices in a lattice $\Lambda$. For the sake of
concreteness, the reader may assume $\Lambda=\mathbb{Z}^n$. Let
$P_0$ be the interior of $P$. For every $q\in \mathbb{N}$, let us
denote by $qP$ the \emph{dilation} of $P$ by a factor $q$. Then
also $qP$ is a polytope with vertices in $\Lambda$, and we denote
its interior by $qP_0$. We define
$$\mathcal{E}_P(q)\doteq |qP\cap\Lambda|$$
and
$$\mathcal{I}_P(q)\doteq |qP_0\cap\Lambda|.$$

A beautiful theorem of Ehrhart states that $\mathcal{E}_P(q)$ is a
polynomial in $q$ of degree $n$ (see \cite{Eh}, \cite{BR}). Then
$\mathcal{E}_P(q)$ is called the \emph{Ehrhart polynomial} of $P$.
Also $\mathcal{I}_P(q)$ is a polynomial. Indeed we have
\begin{equation}\label{IE}
 \mathcal{I}_P(q)=(-1)^n \mathcal{E}_P(-q).
\end{equation}
This important fact is known as \emph{Ehrhart-Macdonald
reciprocity} (see \cite[Theor 4.1]{BR}).

In this paper, we focus on the case when $P$ is a \emph{zonotope}.
Namely, we take a finite list $X$ of vectors  in $\Lambda$. We can
assume that $X$ spans $\mathbb{R}^n$ as a vector space. Then
$$\mathcal{Z}(X)\doteq \left\{\sum_{x\in X}t_x x, 0\leq t_x\leq 1\right\}.$$
is a convex polytope with integer vertices, called the zonotope of $X$.

Zonotopes play a crucial role in several areas of mathematics,
such as hyperplane arrangements, box splines, and partition
functions (see \cite{li}).

To simplify the notation, we set
$\mathcal{E}_X(q)\doteq\mathcal{E}_{\mathcal{Z}(X)}(q)$ and
$\mathcal{I}_X(q)\doteq\mathcal{I}_{\mathcal{Z}(X)}(q)$.

\subsection{Multiplicity Tutte polynomial}

We take $X\subset\Lambda=\mathbb{Z}^n$ as above. For every
$A\subset X$, let $r(A)$ be the rank of $A$, i.e. the dimension of
the spanned subspace of $\mathbb{R}^n$.

The \emph{Tutte polynomial} of $X$, defined in \cite{Tu}, is
$$T_X(x,y)\doteq \sum_{A\subseteq X} (x-1)^{n-r(A)} (y-1)^{|A|-r(A)}.$$

Following \cite{MoT}, we denote by $\langle A\rangle_{\mathbb{Z}}$
and $\langle A\rangle_{\mathbb{R}}$ respectively the sublattice of
$\Lambda$ and the subspace of $\mathbb{R}^n$ spanned by $A$. Let
us define $$\Lambda_A\doteq \Lambda\cap \langle
A\rangle_{\mathbb{R}},$$ the largest sublattice of $\Lambda$ in
which $\langle A\rangle_{\mathbb{Z}}$ has finite index. We define
$m$ as this index:
$$m(A)\doteq \left[\Lambda_A : \langle A\rangle_{\mathbb{Z}}\right].$$
Notice that for every $A\subset X$ of maximal rank, $m(A)$ is
equal to the greatest common divisor of the determinants of the
bases extracted from $A$.

We define the \emph{multiplicity Tutte polynomial} of $X$ as
$$M_X(x,y)\doteq \sum_{A\subseteq X} m(A) (x-1)^{n-r(A)} (y-1)^{|A|-r(A)}.$$

\begin{re}
We say that the list $X$ is \emph{unimodular} if every basis $B$
extracted from $X$ spans $\Lambda$ over $\mathbb{Z}$ (i.e. $B$ has
determinant $\pm 1$). In this case $m(A)=1$ for every $A\subset
X$. Then $M_X(x,y)=T_X(x,y)$.
\end{re}

\section{Theorems}
For every $q\in \mathbb{N}$, let us consider the dilated list
$$qX\doteq  \{qx, x\in X\}$$
and its multiplicity Tutte polynomial $M_{qX}(x,y)$.
We have
\begin{lem}
$$M_{qX}(x,y)=q^n M_{X}\left(\frac{x-1}{q}+1,y\right).$$
\end{lem}

\begin{proof}
By definition
$$M_{qX}(x,y)\doteq \sum_{A\subseteq X} m(qA) (x-1)^{n-r(A)} (y-1)^{|A|-r(A)}.$$
Since clearly $m(qA)=q^{r(A)}m(A)$, this sum equals
$$q^n\sum_{A\subseteq X} m(A) \left(\frac{x-1}{q}\right)^{n-r(A)} (y-1)^{|A|-r(A)}=q^n M_{X}\left(\frac{x-1}{q}+1,y\right).$$
\end{proof}

Now we prove the main result of this paper.

\begin{te}\label{EM}
$$\mathcal{E}_X(q)= q^n M_{X}(1+1/q,1).$$
\end{te}
We give two different proofs of this theorem.
\begin{proof}(I).
By \cite[Prop. 4.5]{MoT}, the number of integer points in the zonotope $\mathcal{Z}(X)$ is equal to $M_X(2,1)$. Since clearly
$q\mathcal{Z}(X)=\mathcal{Z}(qX)$, by the above Lemma we have that
$$\mathcal{E}_X(q)= M_{qX}(2,1)=q^n M_{X}(1+1/q,1).$$
\end{proof}

\begin{proof}(II).
Let $I(X)$ be the family of all linearly independent sublists of $X$.
By \cite[Prop. 2.1]{HHL} (or \cite[Ex. 31]{StL}), we have
\begin{equation}\label{S}
 \mathcal{E}_X(q)=\sum_{A\in I(X)} m(A) q^{|A|}.
\end{equation}
On the other hand we have by definition
 $$M_X(t+1,1)=\sum_{A\in I(X)} m(A) t^{n-|A|}$$
since for every $A\in I(X)$, $r(A)=|A|$.

Hence this polynomial is obtained from the Ehrhart polynomial by reversing its coefficients:
\begin{equation}\label{S2}
M_X(t+1,1) =t^n\mathcal{E}_X(1/t).
\end{equation}
By setting $t=1/q$ we get the claim.
\end{proof}

Notice that the second proof does not rely on \cite[Prop. 4.5]{MoT}, which we can now obtain as a corollary of the above Theorem.

~

By the Theorem above and Formula \ref{IE} we immediately get the
following formula for the number of integer points in the interior
of $q\mathcal{Z}(X)$.
\begin{co}
$$\mathcal{I}_X(q)= (-q)^n M_{X}(1-1/q,1).$$
\end{co}

We also get the following results. The first two were proved in \cite{MoT}, while the third is new.

\begin{co}
~

\begin{enumerate}\label{ECI}
 \item The number $\left| \mathcal{Z}(X)\cap\Lambda \right|$ of integer points in the zonotope is equal to $M_X(2,1)$.
 \item The volume $vol\large(\mathcal{Z}(X)\large)$ of the zonotope is equal to $M_X(1,1)$.
 \item The number $\left| \mathcal{Z}(X)_0\cap\Lambda \right|$ of integer points in the interior of the zonotope is equal to $M_X(0,1)$.
\end{enumerate}
\end{co}

\begin{proof}
The first and the third statement are proved by evaluating at
$q=1$ the polynomials $\mathcal{E}_X(q)$ and $\mathcal{I}_X(q)$
respectively. As for the second claim, we recall that
$vol\large(\mathcal{Z}(X)\large)$ is equal to the leading
coefficient of $\mathcal{E}_X(q)$ (see \cite[Cor. 3.20]{BR}). But
by Formula (\ref{S2}) this is equal to the constant coefficient of
$M_X(1+t,1)$, that is $M_X(1,1)$.
\end{proof}

\section{An example}
Consider the list in $\mathbb{Z}^2$
$$X=\left\{ (3,0), (0,2), (1,1)\right\}.$$
Then
$$M_X(x,y)=(x-1)^2+(3+2+1)(x-1)+(6+3+2)+(y-1)=x^2+4x+y+5,$$
the Ehrhart polynomial is
$$\mathcal{E}_X(q)=11 q^2+6q+1$$
and $$\mathcal{I}_X(q)=11q^2-6q+1.$$ In the picture below we draw
the zonotope $\mathcal{Z}(X)$ and its dilation $2\mathcal{Z}(X)$.
The former contains 18 points, the latter 57. The interiors of
these polytopes contain 6 and 33 points respectively.

\includegraphics[width=80mm]{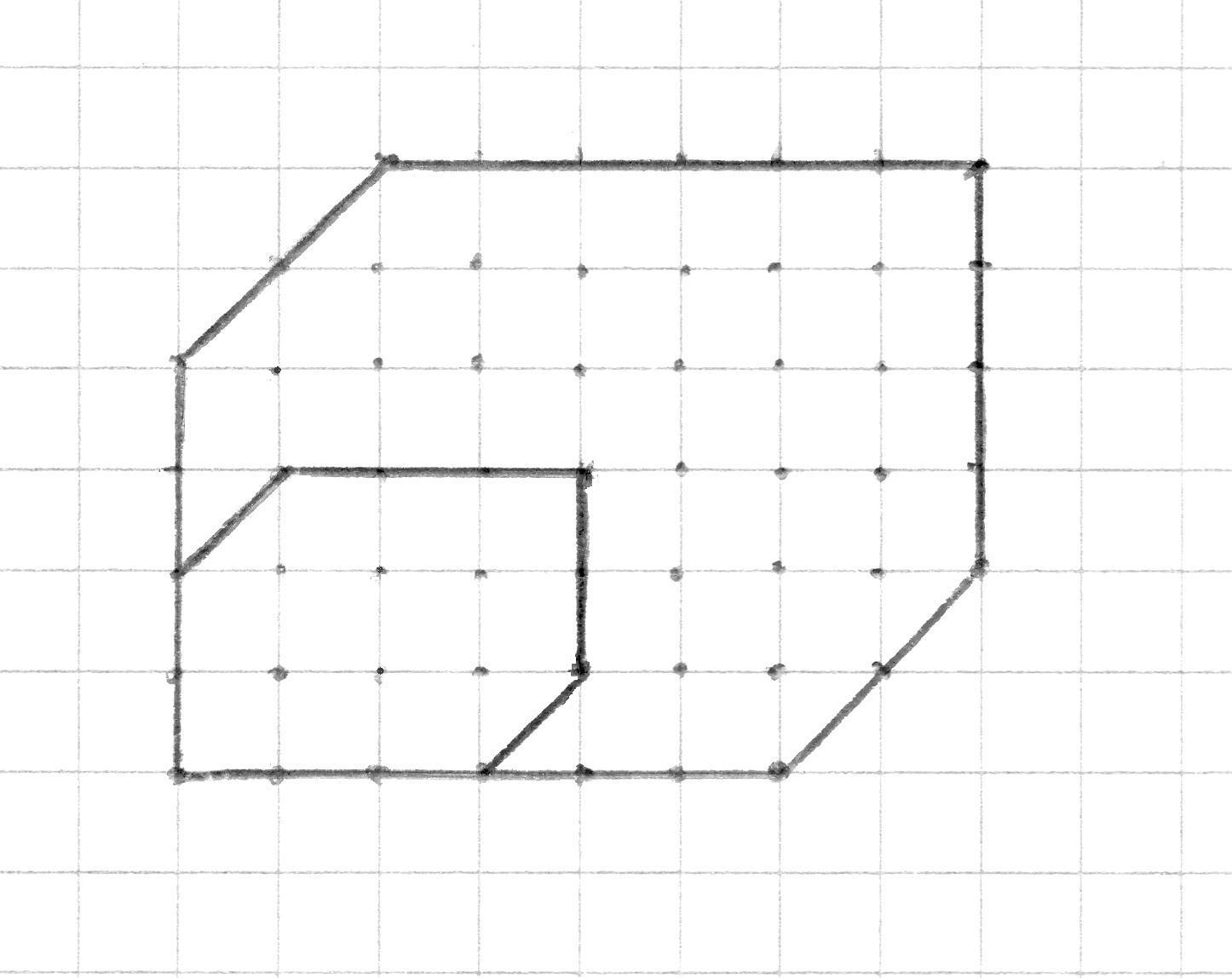}

~

~

\end{document}